\documentclass[reqno,11pt]{amsart}
\usepackage{amsmath,amssymb,amsfonts,amsthm, amscd,indentfirst}
\usepackage{amsmath,latexsym,amssymb,amsmath,
	amscd,amsthm,amsxtra}
\usepackage{color}
\usepackage{pdfpages}
\usepackage{graphics}
\usepackage{enumitem}

\usepackage[usenames,dvipsnames]{xcolor}

\usepackage{varioref}
\usepackage{hyperref}
\usepackage[noabbrev,capitalise,nameinlink]{cleveref}

\hypersetup{colorlinks={true},linkcolor={blue},citecolor=blue}
\usepackage[english]{babel}
\usepackage{csquotes}

\usepackage[backend=biber,
style=numeric,datamodel=ext-eprint,maxbibnames=99]{biblatex}
\DeclareSourcemap{
	\maps[datatype=bibtex]{
		\map{
			\step[fieldsource=pmid, fieldtarget=pubmed]
		}
	}
}
\makeatletter
\DeclareFieldFormat{arxiv}{%
	arXiv\addcolon\space
	\ifhyperref
	{\href{http://arxiv.org/\abx@arxivpath/#1}{%
			\nolinkurl{#1}%
			\iffieldundef{arxivclass}
			{}
			{\addspace\texttt{\mkbibbrackets{\thefield{arxivclass}}}}}}
	{\nolinkurl{#1}
		\iffieldundef{arxivclass}
		{}
		{\addspace\texttt{\mkbibbrackets{\thefield{arxivclass}}}}}}
\makeatother
\DeclareFieldFormat{pmcid}{%
	PMCID\addcolon\space
	\ifhyperref
	{\href{http://www.ncbi.nlm.nih.gov/pmc/articles/#1}{\nolinkurl{#1}}}
	{\nolinkurl{#1}}}
\DeclareFieldFormat{mr}{%
	MR\addcolon\space
	\ifhyperref
	{\href{http://www.ams.org/mathscinet-getitem?mr=MR#1}{\nolinkurl{#1}}}
	{\nolinkurl{#1}}}
\DeclareFieldFormat{zbl}{%
	Zbl\addcolon\space
	\ifhyperref
	{\href{http://zbmath.org/?q=an:#1}{\nolinkurl{#1}}}
	{\nolinkurl{#1}}}
\DeclareFieldAlias{jstor}{eprint:jstor}
\DeclareFieldAlias{hdl}{eprint:hdl}
\DeclareFieldAlias{pubmed}{eprint:pubmed}
\DeclareFieldAlias{googlebooks}{eprint:googlebooks}

\renewbibmacro*{eprint}{%
	\printfield{arxiv}%
	\newunit\newblock
	\printfield{jstor}%
	\newunit\newblock
	\printfield{mr}%
	\newunit\newblock
	\printfield{zbl}%
	\newunit\newblock
	\printfield{hdl}%
	\newunit\newblock
	\printfield{pubmed}%
	\newunit\newblock
	\printfield{pmcid}%
	\newunit\newblock
	\printfield{googlebooks}%
	\newunit\newblock
	\iffieldundef{eprinttype}
	{\printfield{eprint}}
	{\printfield[eprint:\strfield{eprinttype}]{eprint}}}

\usepackage{graphicx}
\usepackage{caption}

\newtheorem{theorem}{Theorem}[section]

\newtheorem{lemma}{Lemma}[section]
\newtheorem{proposition}{Proposition}[section]
\newtheorem{corollary}{Corollary}[section]
\theoremstyle{remark}
\newtheorem{remark}{Remark}[section]

\def\O{\Omega}

\def\S{\Sigma}
\def\n{\nabla}

\def\p{\partial}

\def\n{\nabla}

\def\O{\Omega}
\def\p{\partial}

\def\De{\Delta}
\def\n{\nabla}
\def\<{\langle}
\def\>{\rangle}

\def\De{\Delta}
\def\n{\nabla}

\def\SS{\mathbb{S}}

\def\O{\Omega}
\def\p{\partial}

\def\rr{\mathbb{R}}

\begin{document}
	
	\title[A Heintze-Karcher type inequality]{A Heintze-Karcher type inequality for hypersurfaces with capillary boundary}
	\author{Xiaohan Jia}
	\address{School of Mathematical Sciences\\
		Xiamen University\\
		361005, Xiamen, P.R. China}
	\email{jiaxiaohan@xmu.edu.cn}
	\author{Chao Xia}
	\address{School of Mathematical Sciences\\
		Xiamen University\\
		361005, Xiamen, P.R. China}
	\email{chaoxia@xmu.edu.cn}
	\author{Xuwen Zhang}
	\address{School of Mathematical Sciences\\
		Xiamen University\\
		361005, Xiamen, P.R. China}
	\email{xuwenzhang@stu.xmu.edu.cn}
	\thanks{This work is  supported by the NSFC (Grant No. 11871406, 12126102).
	}
	
	\begin{abstract}
		In this paper, we establish a Heintze-Karcher type inequality for hypersurfaces with capillary boundary of contact angle $\theta\in (0,\frac{\pi}{2})$ in a half space or a half ball, by using solution to a mixed boundary value problem in Reilly type formula. Consequently, we give a new proof of Alexandrov type theorem for embedded capillary constant mean curvature hypersurfaces with contact angle $\theta\in (0,\frac{\pi}{2})$ in a half space or a half ball. 		
		\
		
		
	\end{abstract}
	
	\maketitle
	
	\medskip
	

	\section{Introduction}
	Constant mean curvature (CMC) hypersurfaces play as stationary points of the isoperimetric problem.
	A celebrated theorem in differential geometry, which is known as Alexandrov's theorem, says that  any embedded closed hypersurface in $\rr^{n+1}$ with constant mean curvature is a geodesic sphere. It was first proved by Alexandrov \cite{Aleksandrov62} by nowadays famous Alexandrov reflection method or moving plane method. 
	
	Reilly \cite{Reilly77} found an integral formula, which is known as Reilly's formula, and gave an integral proof of Alexandrov's theorem. Later, inspired by Reilly's proof and a volume estimate by Heintze-Karcher \cite{HK78}, Ros \cite{Ros87} established a geometric inequality, known today as the Heintze-Karcher inequality, which says that for a bounded domain $\O$ in $\rr^{n+1}$ with smooth and mean convex boundary $\S=\p\O$, it holds that 
	\begin{eqnarray}\label{hk}
		\int_\S\frac{1}{H}\ge  \frac{n+1}{n}|\O|,
	\end{eqnarray}
	with equality holding if and only if $\S$ is a geodesic sphere. Here $H$ is the mean curvature of $\S$ and $\S$ is said to be mean convex if $H>0$. 
	A combination of the Heintze-Karcher inequality \eqref{hk} and the Minkowski-Hsiung formula yields Alexandrov's theorem.
	
	Motivated by the study of the equilibrium shapes of a liquid confined in a given container or a pendent drop, a corresponding relative isoperimetric problem has been intensively studied, see for example \cite[Chapter 19]{Mag12}. In particular, capillary hypersurfaces play as stationary points of the relative isoperimetric problem. A capillary hypersurface in a container is a hypersurface intersecting the boundary of the container at a constant contact angle. Alexandrov type theorems are established in the spirit of the Alexandrov reflection method when the container possesses certain symmetry. Wente \cite{Wente80} showed that any embedded capillary hypersurface in the half space $$\rr^{n+1}_+=\{x\in \rr^{n+1}:  x_{n+1}>0\}$$ with boundary on the hyperplane $\p \rr^{n+1}_+$, whose mean curvature function is axially symmetric, must be axially symmetric. Wente used a maximum principle with corner due to Serrin \cite{Serrin71} to run the Alexandrov reflection method in such situation. Similarly, Athanassenas \cite{Ath87} and Ros-Souam \cite{RS97} obtained Alexandrov type theorems in a slab or a ball.
	
	It is interesting to see whether one can establish Alexandrov type theorems for capillary hypersurfaces via integral method in the spirit of Reilly \cite{Reilly77} and Ros \cite{Ros87}. 
	In a very recent work by the second author and Wang \cite{WX19}, The following Heintze-Karcher type inequality has been derived. For an embedded mean convex free boundary hypersurface $\S$ lying in the half ball $$\mathbb{B}_+^{n+1}=\{x\in \rr^{n+1}: |x|<1, \quad x_{n+1}>0\},$$ it holds that 
	\begin{eqnarray}\label{hk1}
		\int_\S\frac{x_{n+1}}{H}dA\ge\frac{n+1}{n}\int_\O x_{n+1}dx,
	\end{eqnarray}
	with equality holding if and only if $\S$ is a spherical cap with free boundary. Here $\O$ is the enclosed domain by $\S$ and $\p \mathbb{B}_+^{n+1}$ and a free boundary hypersurface is a hypersurface that intersects $\SS^n$ orthogonally. The proof is based on a mixed boundary value problem in $\O$ and a Reilly type formula due to the second author and Qiu \cite{QX15}. By combining \eqref{hk1} and a new Minkowski-Hsiung's formula \cite{WX19}, they reproved Alexandrov type theorem for free boundary CMC hypersurface in a half ball. We would like to mention that Heintze-Karcher type inequality for free boundary hypersurfaces in a cone and in a wedge have been derived by Choe-Park \cite{CP11} and Lopez \cite{Lopez14} respectively.
	
	The purpose of this paper is to establish a Heintze-Karcher type inequality for capillary hypersurfaces in a half space or a ball with contact angle $\theta\in (0, \frac{\pi}{2})$, in order to give integral proof of corresponding Alexandrov type theorem. In the following, we call $\S$ is a $\theta$-capillary hypersurface in some container $B$ for $\theta\in (0,\pi)$ if  it intersects $\p B$ at contact angle $\theta$.
	
	\begin{theorem}[Heintze-Karcher type inequality in a half space]\label{Thm-HK-halfspace}
		Let  $\theta\in(0,\frac{\pi}{2}]$ and  $\S\subset \bar \rr^{n+1}_+$ be a smooth, compact, embedded, mean convex $\theta$-capillary hypersurface. Let $\O$ be the enclosed domain by $\S$ and $\p \rr^{n+1}_+$. Then it holds, 
		\begin{align}\label{EQ-HK-halfspace}
			\int_\S\frac{1}{H} dA\ge \frac{n+1}{n} |\O|+\cos\theta\frac{\left(\int_\S \<\nu, E_{n+1}\>dA\right)^2}{\int_\S H\<\nu, E_{n+1}\> dA},
		\end{align}
		where $\nu$ is the outward unit normal to $\S$ (with respect to $\O$) and $E_{n+1}=(0,\cdots, 0, 1)$.
		Equivalently, 
		\begin{align}\label{EQ-HK-halfspace2}
			\int_\S\frac{1}{H} dA\ge \frac{n+1}{n} |\O| +\cot\theta\frac{|\p\O\cap \p\rr^{n+1}_+|^2}{|\p\S|}.
		\end{align}				
		Equality in \eqref{EQ-HK-halfspace} or \eqref{EQ-HK-halfspace2} holds if and only if $\S$ is a $\theta$-capillary spherical cap.
	\end{theorem}
	
	\begin{theorem}[Heintze-Karcher type inequality in a half ball]\label{Thm-HK-halfball}
		Let  $\theta\in(0,\frac{\pi}{2}]$ and  $\S\subset \bar{\mathbb{B}}^{n+1}_+$ be a smooth, compact, embedded, mean convex $\theta$-capillary hypersurface. Let $\O$ be the enclosed domain by $\S$ and $\SS^n$. Then it holds, 
		\begin{align}\label{EQ-HK-halfball}
			\int_\S\frac{x_{n+1}}{H} dA\ge\frac{n+1}{n} \int_\O x_{n+1}dx -\cos\theta\frac{\left(\int_\S \<\nu, E_{n+1}\>dA\right)^2}{\int_\S H\<\nu, E_{n+1}\> dA},
		\end{align}
		where $\nu$ is the outward unit normal to $\S$ (with respect to $\O$) and $E_{n+1}=(0,\cdots, 0, 1)$.
		Equivalently, 
		\begin{align}\label{EQ-HK-halfball2}
			\int_\S\frac{x_{n+1}}{H} dA\ge\frac{n+1}{n} \int_\O x_{n+1}dx +\cos\theta\frac{\left(\int_{\p\O\cap \SS^n}x_{n+1}dA\right)^2}{ \int_{\p\S} \<\mu, E_{n+1}\>ds}.
		\end{align}	
		where $\nu$ is the outward unit conormal to $\p\S$ (with respect to $\S$).			
		Equality in \eqref{EQ-HK-halfball} or \eqref{EQ-HK-halfball2} holds if and only if $\S$ is a $\theta$-capillary spherical cap.		
	\end{theorem}
	
	When $\theta=\frac{\pi}{2}$, the Heintze-Karcher type inequality in $\rr^{n+1}_+$ is a simple consequence of the classical one \eqref{hk} since one could make a mirror reflection of $\S$ through $\p \rr^{n+1}_+$ to get a closed mean convex hypersurface in $\rr^{n+1}$, while the Heintze-Karcher type inequality in $\mathbb{B}^{n+1}_+$ is exactly \eqref{hk1} proved in \cite{WX19}.
	Up to our knowledge, the Heintze-Karcher type inequality in $\rr^{n+1}_+$ or $\mathbb{B}^{n+1}_+$ is new for capillary hypersurfaces with  contact angle other than $\frac{\pi}{2}$.
	
	The Minkowski type formula for a $\theta$-capillary hypersurface $\S\subset \bar \rr^{n+1}_+$ is 
	\begin{align}\label{mink-halfspace3}
		\int_\S n(1-\cos\theta\<\nu, E_{n+1}\>) -H\<x,\nu\>dA=0,\end{align}
	and for $\S\subset \bar{\mathbb{B}}^{n+1}_+$ is
	\begin{align}\label{mink-halfspace3}
		\int_\S n(x_{n+1}+\cos\theta\<\nu, E_{n+1}\>) -H\<X_{n+1},\nu\>dA=0,\end{align}
	where 
	\begin{align}\label{Xconf}
	X_{n+1}=x_{n+1}x-\frac12(|x|^2+1)E_{n+1},\end{align}
	 see \cite{AS16, WX19}.
	As a consequence of the Heintze-Karcher type inequality and the Minkowski type formula, we have the following Alexandrov type theorems.	
	\begin{corollary}\label{Theorem2-halfspace}
		Let $\theta_i\in(0,\frac{\pi}{2}]$ and $\S\subset \bar \rr^{n+1}_+$ be a smooth, compact, embedded constant mean curvature $\theta$-capillary hypersurface. Then
		\begin{align}\label{EQ-HK-halfspace3}
			\int_\S\frac{n(1-\cos\theta\<\nu, E_{n+1})\>}{H}dA\ge (n+1)|\O|.		\end{align}
		Moreover, $\S$ is a $\theta$-capillary spherical cap.\end{corollary}
	\begin{corollary}\label{Theorem2-ball}
		Let $\theta_i\in(0,\frac{\pi}{2}]$ and $\S\subset \bar{\mathbb{B}}^{n+1}_+$ be a smooth, compact, embedded constant mean curvature $\theta$-capillary hypersurface. Then
		\begin{align}\label{EQ-HK-halfball3}
			\int_\S\frac{n(x_{n+1}+\cos\theta\<\nu, E_{n+1})\>}{H}dA\ge (n+1)\int_\O x_{n+1}dx.		\end{align}
		Moreover, $\S$ is a $\theta$-capillary spherical cap.\end{corollary}
	
		In view of the Minkowski type formulas,  we conjectured that \eqref{EQ-HK-halfspace3} and \eqref{EQ-HK-halfball3} holds for any compact, embedded $\theta$-capillary hypersurface in $\bar \rr^{n+1}_+$ and $\bar{\mathbb{B}}^{n+1}_+$ respectively\footnote{The conjecture will be solved by using different method in the forthcoming paper \cite{JWXZ22}}.
	
	Our proof of the the Heintze-Karcher type inequality is motivated from that by Ros \cite{Ros87} and Wang-Xia \cite{WX19}. We illustrate the idea in the case of $\rr^{n+1}_+$. The key point is to use the solution of the following mixed boundary value problem in $\O$:
	\begin{align}\label{MixedBdryPb-halfspace}
		\begin{cases}
			\bar\Delta f=1 \quad&\text{in }\O,\\
			f=0 \quad&\text{on }\S,\\
			\bar\n_{\bar N} f=c &\text{on }\p \O\cap \p \rr^{n+1}_+,	\end{cases}	
	\end{align}
	where $\bar N=-E_{n+1}$ is the downward unit normal to $\p \rr^{n+1}_+$ and $c$ is a suitable chosen constant defined in \eqref{c1}.
	The existence of a unique solution $f\in C^\alpha(\bar \O)\cap C^\infty(\O)$ to \eqref{MixedBdryPb-halfspace} has been shown by Lieberman \cite{Liebermann86} by using Perron's method. Moreover, it has been proved by Lieberman \cite{Liebermann89} that $f\in C^{1,\alpha}(\bar \O)$ for some $\alpha\in (0, 1)$ and $\bar \n^2 f\in L^2(\O)$ provided $\theta<\frac{\pi}{2}$. We will review the existence and regularity  in \cref{Sec3} and \cref{Appendix}.
	The regularity enables us to use the solution $f$ to \eqref{MixedBdryPb-halfspace} in the classical Reilly formula. After carefully rearranging term and integration by parts, we arrive at the Heintze-Karcher type inequality.
	
	We remark that our proof also works for the contact angle $\theta\in (\frac{\pi}{2}, \pi)$,
	except for the lack of regularity for \eqref{MixedBdryPb-halfspace}.
	In this case, Lieberman's optimal regularity result for general boundary data for the mixed boundary value problem is not enough for applying Reilly's formula. Nevertheless, we expect that the regularity can be improved for the special boundary data in \eqref{MixedBdryPb-halfspace}, that is,
	a vanishing Dirichlet condition on one part of the boundary and a constant Neumann condition on the other part, in order to apply Reilly's formula to get the 
	Heintze-Karcher type inequality for $\theta\in (\frac{\pi}{2}, \pi)$.
	
	\section{Preliminaries}\label{Sec2}

	Let $B$ be either $\bar \rr^{n+1}_+$ or $\bar{\mathbb{B}}^{n+1}$. Let $\O$ be a bounded domain in $B$ with piecewise smooth boundary $\p\O=\S\cup T$, where $\S=\overline{\p\O\cap B}$ is a smooth compact embedded $\theta$-capillary hypersurface in $\bar B$ and $T=\p\O\cap \p B$. 
	Denote the corner by $\Gamma=\S\cap T=\p\S=\p T$, which is a smooth co-dimension two submanifold in $\rr^{n+1}$. We use the following notation for normal vector fields. Let $\nu$ and $\bar N$ be the outward unit normal to $\S$ and $\p B$ (with respect to $\O$) respectively. Let $\mu$ be the outward unit co-normal to $\Gamma=\p\S\subset \S$ and $\bar \nu$ be the outward unit co-normal to $\Gamma=\p T\subset T$. Under this convention, along $\Gamma$ the bases $\{\nu,\mu\}$ and $\{\bar \nu,\bar N\}$ have the same orientation in the normal bundle of $\p \S\subset \rr^{n+1}$. In particular, $\S$ is $\theta$-capillary if along $\Gamma$, \begin{eqnarray}\label{munu}
		&&\mu=\sin \theta \bar N+\cos\theta \bar \nu,\quad \nu=-\cos \theta \bar N+\sin \theta \bar \nu.
	\end{eqnarray}
	
	We denote by $\bar \n$, $\bar \Delta$, $\bar \n^2$ and $\bar{\rm div}$, the gradient, the Laplacian, the Hessian and the divergence on $\rr^{n+1}$ respectively, while by $\n$, $\Delta$, $\n^2$ and ${\rm div}$, the gradient, the Laplacian, the Hessian and the divergence on the smooth part of $\p \O$, respectively. Let $g$, $h$ and $H$ be the first, second fundamental forms and the mean curvature  of the smooth part of $\p \O$ respectively. Precisely, $h(X, Y)=\<\bar\n_X\nu, Y\>$ and $H={\rm tr}_g(h)$. 
	
	We need the following version of Reilly's formula.
	\begin{proposition}
		Let $\O$ be a bounded domain in $B$  
		with piecewise smooth boundary $\p\O$ as above. Let $f\in C^\infty(\bar\O\setminus \Gamma)$ such that $\bar\n^2 f\in L^2(\O)$. Then we have 
		\begin{eqnarray}\label{reilly}
			&&\int_\O  (\bar\Delta f)^2-\left|\bar\n^2 f\right|^2 dx
			\\&=&\int_{\p\O}  f_\nu \De f-g(\n f_\nu, \n f) +H f_\nu^2 +h(\n f, \n f)dA. \nonumber
		\end{eqnarray}
		Denote $V=x_{n+1}$ and assume $V\ge 0$ in $\O$. Then
		\begin{eqnarray}\label{qx}
			&&\int_\O  V\left((\bar\Delta f)^2-\left|\bar\n^2 f\right|^2\right) dx 
			\\&=&\int_{\p\O} V\left(f_\nu-\frac{V_\nu}{V}f\right)\left(\De f-\frac{\De V}{V}f\right)dA \nonumber\\&& - \int_{\p\O} Vg\left(\n \left(f_\nu-\frac{V_\nu}{V}f\right), \n f-\frac{\n V}{V}f\right)dA 
			\nonumber\\&&+\int_{\p\O} VH\left(f_\nu-\frac{V_\nu}{V}f\right)^2 +\left(h-\frac{V_{\nu}}{V}g\right)\left(\n f-\frac{\n V}{V}f, \n f-\frac{\n V}{V}f\right)dA.\nonumber
		\end{eqnarray}
	\end{proposition}
	\begin{proof}
		Let $Y=\bar\Delta f\bar \n f-\frac12\bar \n |\bar \n f|^2$. Then $$\bar{\rm div}(Y)= (\bar\Delta f)^2-\left|\bar\n^2 f\right|^2,$$
		By the assumption $\bar\n^2 f\in L^2(\O)$, we know that
		$Y\in L^2(\O)$ and $\bar{\rm div}(Y)\in L^1(\O)$.
		It follows from the divergence theorem in \cite[Lemma 2.1]{PT20} that 
		$$\int_\O \bar{\rm div}(Y) dx=\int_{\p\O} \<Y, \nu_{\p\O}\>dA.$$
		It is now a standard argument by using Gauss-Weingarten formula to rearrange the boundary integrand as in Reilly \cite{Reilly77} so as to get \eqref{reilly}.
		Formula \eqref{qx} follows similarly as in Qiu-Xia \cite{QX15} and Li-Xia \cite{LX19}, taking account of the divergence theorem in \cite[Lemma 2.1]{PT20} as before.
	\end{proof}
	
	Next we need the following structural lemma for compact hypersurfaces in $\rr^{n+1}$ with boundary, these results are well-known and widely used, see \cite{AS16, Souam21}.
	
	\begin{lemma}\label{PropAS16-2.4}
		Let $\S\subset \rr^{n+1}$ be a smooth compact hypersurface with boundary. Then it holds that
		\begin{align}\label{AS16-(2.3)}
			n\int_{\S}\nu dA=\int_{\p \S}\left\{\left<x,\mu\right>\nu-\left<x,\nu\right>\mu\right\}ds.
		\end{align}
	\end{lemma}
	\begin{proof}
		Let $Z=\left<\nu, e\right>x^T- \left<x,\nu\right>e^T$ for some constant vector $e$, where $x^T$ and $e^T$ denote the tangential component of $x$ and $e$ respectively. One computes that
		$${\rm div}(Z)=n\left<\nu, e\right>.$$
		Integration by parts yields the assertion.
	\end{proof}
	
	\section{Mixed boundary value problem}\label{Sec3}
	
	In this section, we collect the existence and regularity results for two mixed boundary problems which will be used in the proof of the main theorems.
	
	\subsection{Half space}
	Let $B=\rr^{n+1}_+$, consider the following mixed boundary value problem in $\O\subset \rr^{n+1}_+$:
	\begin{align}\label{MixedBdryPb-halfspace2}
		\begin{cases}
			\bar\Delta f=1 \quad&\text{in }\O,\\
			f=0 \quad&\text{on }\S,\\
			\bar\n_{\bar N} f=c &\text{on } {\rm int}(T).	\end{cases}	
	\end{align}
	\begin{theorem}
		Assume $\theta\in (0,\frac{\pi}{2})$. There exists a unique solution $f\in C^{\infty}(\bar \O\setminus \Gamma)\cap C^{1,\alpha}(\bar \O)$ to \eqref{MixedBdryPb-halfspace2} such that $|\bar\n^2 f|\le Cd_\Gamma^{-\beta}$ for some $\beta\in (0, 1)$ and some constant $C>0$. Here $d_\Gamma$ is the distance function to $\Gamma$. 
	\end{theorem}

\begin{proof}	
{\color{black}Notice that the Fredholm alternative \cref{fredholm lemma} is applicable for \eqref{MixedBdryPb-halfspace2} and we proceed by rulling out case (a).}

The maximum principle (see for example \cite[Lemma 4.1]{Tang13}) shows that the homogeneous problem
	\begin{align*}
	    \begin{cases}
	        \bar\Delta f=0 \quad&\text{in }\O,\\
			f=0 \quad&\text{on }\S,\\
			\bar\n_{\bar N} f=0 &\text{on } {\rm int}(T),
	    \end{cases}
	    \end{align*}
     has only the trivial solution, this assures the validity of case $(b)$ in \cref{fredholm lemma}. 
     {\color{black}By \cref{remark-appendix}}, we thus obtain a unique solution $f\in C^{1,\alpha}(\bar \O)$ solving \eqref{MixedBdryPb-halfspace2}.
     
     The regularity $f\in C^{\infty}(\bar \O\setminus \Gamma)$ follows from the classical regularity theory for elliptic equation and the decay estimate $\vert {\bar\nabla^2 f\vert}\leq Cd_\Gamma^{-\beta}$ with $\beta=\frac{3-\lambda}{2}\in (0, 1)$ follows directly from \cref{regularity thm} and \eqref{c0 estimate} by choosing $\lambda\in (1, \min\{3,\frac{\pi}{2\theta}\})$ and $a=\frac{\lambda+3}{2}>2$.
     \end{proof}

\subsection{Half ball}	
	Let $B=\mathbb{B}^{n+1}$ and $V=x_{n+1}$, consider the following mixed boundary value problem in $\O\subset \mathbb{B}^{n+1}_+$:
	\begin{align}\label{MixedBdryPb-halfball2}
		\begin{cases}
			\bar\Delta f=1 \quad&\text{in }\O,\\
			f=0 \quad&\text{on }\S,\\
			\bar\n_{\bar N} f-f=c &\text{on } {\rm int}(T).	\end{cases}	
	\end{align}
	\begin{theorem}
		Assume $\theta\in (0,\frac{\pi}{2})$. There exists a unique solution $f\in C^{\infty}(\bar \O\setminus \Gamma)\cap C^{1,\alpha}(\bar \O)$ to \eqref{MixedBdryPb-halfball2} such that $|\bar\n^2 f|\le Cd_\Gamma^{-\beta}$ for some $\beta\in (0, 1)$ and some constant $C>0$. Here $d_\Gamma$ is the distance function to $\Gamma$. 
	\end{theorem}
	\begin{proof}
	{\color{black}We begin by noticing that the Fredholm alternative(\cref{fredholm lemma}) applies for \eqref{MixedBdryPb-halfball2}, and we will rull out case $(a)$ by virtue of two maximum principles.}
	
	 Presicely, let $f$ be the solution of the following homogeneous problem
	\begin{align}\label{homog MBVP-halfball}
	    \begin{cases}
	        \bar\Delta f=0 \quad&\text{in }\O,\\
			f=0 \quad&\text{on }\S,\\
			\bar\n_{\bar N} f-f=0 &\text{on } {\rm int}(T),
	    \end{cases}
	    \end{align}
	  by applying the maximum principle(\cite[Proposition 2.3]{GX19})\footnote{\color{black}The validity of such maximum principle is based on the geometry of the domain $B_+^{n+1}$. Precisely, we note that the first Mixed Robin-Dirichlet eigenvalue of the homogeneous problem plays an important role in the proof. We refer to \cite{GX19} for a detailed account.}to $f$, we see that either $f\equiv 0$ in $\O$ or $f<0$ in $\O\cup T$.
	  
	  On the other hand, by the maximum principle for mixed boundary equation(see for example \cite[Lemma 4.1]{Tang13}) for $-f$, we see that $-f\leq0$. Hence, the homogeneous problem \eqref{homog MBVP-halfball} has only the trivial solution $f\equiv 0$, {\color{black}which assures the validity of case $(b)$ in \cref{fredholm lemma}.}
     Thus, we have found a unique solution $f\in C^{1,\alpha}(\bar \O)$ solving \eqref{MixedBdryPb-halfball2}, due to \cref{remark-appendix}. 
     
     The regularity statements follow similarly as above. 
     \end{proof}
	

	\section{The case of half space}\label{Sec4}
	In this section we let $\O\subset\rr_+^{n+1}$. In this case, $\bar N=-E_{n+1}$. We collect the following integral identities, which has been proved in \cite{AS16}.
	\begin{proposition}\label{PropMinkowski-Poly} It holds that
		\begin{align}
			&\int_{\S}\left<\nu, E_{n+1}\right>dA=|T|, \label{Conservation-halfspace}\\
			&\int_{\S}H \left<\nu, E_{n+1}\right>dA=\sin\theta|\Gamma|,\label{Balancing-halfspace}\\
			& \int_\S n(1-\cos\theta\<\nu, E_{n+1}\>) -H\<x,\nu\>dA=0.\label{Minkowski-Poly}
		\end{align} 	
	\end{proposition}
	\begin{proof} We prove it for completeness.
		Since $$\bar{\rm div} E_{n+1}=0,$$ by using integration by parts in $\O$, we get \eqref{Conservation-halfspace}.
		Since $${\rm div}(E_{n+1}^T)=-H\left<\nu, E_{n+1}\right>,$$ by using integration by parts in $\S$ and \eqref{munu}, we get \eqref{Balancing-halfspace}.
		Since $${\rm div}(x^T)=n-H\left<x, \nu\right>,$$ by using integration by parts in $\S$, we get
		\begin{align}\label{xxeq1}
			\int_\S n -H\<x,\nu\>dA=\int_{\Gamma}\<x, \mu\> ds.
		\end{align} 
		On the other hand, by \eqref{AS16-(2.3)} and \eqref{munu}, we get
		\begin{align}\label{xxeq2}
			\int_\S \<\nu, E_{n+1}\> dA&=\int_{\Gamma}\left(\left<x,\mu\right>\<\nu, E_{n+1}\>-\left<x,\nu\right>\<\mu, E_{n+1}\>\right) ds\nonumber
			\\&=\int_{\Gamma}\frac{1}{\cos\theta}\left<x,\mu\right> ds.
		\end{align} 
		Identity \eqref{Minkowski-Poly} follows from \eqref{xxeq1} and \eqref{xxeq2}.
	\end{proof}
\begin{proof}[{\bf Proof of \cref{Thm-HK-halfspace}.}]
 Let $f$ be the solution to \eqref{MixedBdryPb-halfspace2} with 
	\begin{eqnarray}\label{c1}
		c=-\frac{n}{n+1}\cot\theta\frac{  |T|}{|\Gamma|}.
	\end{eqnarray}
	
At each $p\in \Gamma$, we find some neighborhood $B_\delta(p)$ with $\delta>0$ small, such that up to a coordinate transformation, there exists a Cartesian coordinate $(x^1,\ldots,x^{n+1})$ in $B_\delta(p)\cap \bar \O$, such that the {\color{black}corresponding} cylindrical coordinates $(r,\eta,x')$ is given by 
 $$ x_1=r\cos\eta,\quad  x_2=r\sin\eta,\quad x'=(x^3,\ldots,x^{n+1}),$$
 with $r\in [0,\delta)$ and $\eta\in [0,\theta]$.
 	Since $|\bar\n^2 f|\le Cd_\Gamma^{-\beta}$ for $\beta\in (0, 1)$, one see that 
	$$\int_{\O} |\bar\n^2 f|^2 dx\le C_1+ C_2\int_0^\theta\int_0^\delta r^{-2\beta} r dr d\eta<+\infty.$$
	Also since $\n^2 f(X, Y)=\bar \n^2 f(X, Y)+ h(X, Y)\bar\nabla_{\bar N} f$ on $T$ for some tangential vector fields $X, Y$, we see that $$\int_{T} |\n^2 f| dA\le \int_{T} |\bar \n^2 f|+ C_3|\bar \n f| dA\le C_4\int_0^\delta r^{-\beta}  dr +C_5<+\infty.$$
	That means $|\bar\n^2 f|\in L^2(\O)$ and $|\n^2 f|\in L^1(T)$.
	
	In particular, Reilly's formula \eqref{reilly} is applicable for $f$ and we get
	\begin{eqnarray}\label{xeq1}
		\frac{n}{n+1}|\O|=\frac{n}{n+1}\int_\O (\bar\De f)^2 dx\ge \int_{\S} Hf_\nu^2 dA+ c\int_{T} \De f dA.
	\end{eqnarray}
	Here we used the Cauchy-Schwarz inequality, the boundary condition $f=0$ on $\S$ and $\bar\n_{\bar N} f=c$ on $T$, as well as the fact that $h^T\equiv 0$.
	
	Since $|\n^2 f|\in L^1(T)$, one uses the divergence theorem and the dominate convergence theorem to get
	\begin{eqnarray}\label{xeq2}
		\int_{T} \De f dA=\int_{\Gamma} \<\n^T f, \bar \nu\> ds,
	\end{eqnarray}
	where $\n^T$ denotes the covariant derivative on $T$.
	Next, since $f\in C^1(\bar \O)$, by using \eqref{munu}, we get 
	\begin{eqnarray}\label{xeq3}
		\int_{\Gamma} \<\n^T f, \bar \nu\> ds&=&\int_{\Gamma} \<\n^T f, \frac{1}{\cos\theta} \mu\> ds\nonumber
		\\&=&\frac{1}{\cos\theta} \int_{\Gamma}\left<\bar\nabla f-\left<\bar\nabla f,\bar N\right>\bar N,  \mu\right>ds\nonumber
		\\&=&\frac{1}{\cos\theta} \int_{\Gamma}  \bar\nabla_\mu f -  \sin\theta \bar\nabla_{\bar N} fds\nonumber
		\\&=&-c \tan\theta|\Gamma|.
	\end{eqnarray}
	In the last equality we used $f=0$ on $\S$ which implies $\bar\nabla_\mu f=0$ along $\Gamma$, and $\bar\nabla_{\bar N} f=c$ on $T$.
	
	Combining \eqref{xeq1} - \eqref{xeq3}, we arrive at
	\begin{eqnarray}\label{xeq4}
		\frac{n}{n+1}|\O|\ge\int_{\S} Hf_\nu^2 dA -c^2\tan\theta |\Gamma|.
	\end{eqnarray}
	
	On the other hand, by using the divergence theorem \cite[Lemma 2.1]{PT20}, 
	\begin{eqnarray}\label{xeq5}
		|\O|=\int_\O \bar \Delta f dx =\int_\S f_\nu dA+\int_T f_{\bar N}dA=\int_\S f_\nu dA+c|T|.	\end{eqnarray}
	By the choice of $c$ in \eqref{c1}, we see that
	\begin{eqnarray}\label{xeq6}
		\int_{\S} f_\nu dA=  |\O|{+}\frac{n}{n+1}\cot\theta\frac{|T|^2}{|\Gamma|} \ge\frac{n+1}{n}\int_{\S} Hf_\nu^2 dA.
	\end{eqnarray}
	Using the H\"older inequality $$ \left(\int_{\S} f_\nu dA\right)^2\le \int_{\S} Hf_\nu^2 dA\int_{\S} \frac{1}{H} dA,$$ in \eqref{xeq6}, we obtain
	\begin{eqnarray*}
		\int_{\S} \frac{1}{H} dA\ge \frac{n}{n+1} |\O|+\cot\theta\frac{|T|^2}{|\Gamma|}.
	\end{eqnarray*}
	This is exactly \eqref{EQ-HK-halfspace2}. Inequality \eqref{EQ-HK-halfspace} is equivalent to \eqref{EQ-HK-halfspace2} in view of \eqref{Conservation-halfspace} and \eqref{Balancing-halfspace}.
	
	We are thus left to characterize the equality case in \eqref{EQ-HK-halfspace} or \eqref{EQ-HK-halfspace2}.
	Indeed, if equalities hold throughout the argument above, from the Cauchy-Schwarz inequality for $\vert\bar\nabla^2 f\vert^2$ we know that $\bar\nabla^2f$ is proportional to the metric $g_{\rm euc}$, i.e., 
	\begin{align}\label{proportional}
		\bar\nabla^2f=\frac{1}{n+1}g_{\rm euc}.
	\end{align}
	 Restricting \eqref{proportional} at $x\in\Sigma$ and using $f\mid_{\Sigma}=0$, we see
	\begin{align}\label{g_ijandh_ij}
		h_{ij}(x)f_{\nu}(x)=\frac{1}{n+1}g_{ij}
	\end{align}
which implies that $\Sigma$ is an umbilical hypersurface in $\mathbb{R}^{n+1}_+$. In particualr, it is well-known that umbilical hypersurfaces in the Euclidean space must be spherical, which completes the proof.
	
\end{proof}

Finally, with the Heintze-Karcher inequality \eqref{EQ-HK-halfspace} in force, we are in the position to prove the Alexandrov type theorem.

\begin{proof}[{\bf Proof of \cref{Theorem2-halfspace}}]
	Since $\Sigma$ is now a CMC $\theta$-capillary hypersurface, the Heintze-Karcher type inequality \eqref{EQ-HK-halfspace} reads
		\begin{align}\label{eq-Corollary1-0}
		\int_\S\frac{1}{H} dA\ge \frac{n+1}{n} |\O|+\frac{\cos\theta}{H}\left(\int_\S \<\nu, E_{n+1}\>dA\right).
	\end{align}
On the other hand, since $H$ is a positive constant on $\Sigma$, the Minkowski type formula \eqref{Minkowski-Poly} shows that \eqref{eq-Corollary1-0} is indeed an equality. By the characterization of equality case in \cref{Thm-HK-halfspace}, $\Sigma$ must be a $\theta$-capillary spherical cap. This completes the proof.

\end{proof}

	\section{The case of half ball}\label{Sec5}
	In this section we let $\O\subset\rr_+^{n+1}$. In this case, $\bar N(x)=x$. We collect the following integral identities, which have been proved in \cite[Proposition 3.2]{WX19}. For the sake of simplicity, in this section, the smooth function $x_{n+1}$ and the vector field $X_{n+1}$ in \eqref{Xconf} are abbreviated by $V$ and $X$. Recall that $X$ is a conformal Killing vector field with $$\frac12(\bar \n_iX_j+\bar \n_j X_i)=V\delta_{ij}.$$
		\begin{proposition}\label{PropMinkowski-Ball} It holds that
		\begin{align}
			&\int_{\S}\left<\nu, E_{n+1}\right>dA=-\int_TVdA, \label{Conservation-Ball}\\
			&\int_{\S}H \left<\nu, E_{n+1}\right>dA=-\int_\Gamma\left<\mu,E_{n+1}\right>ds,\label{Balancing-Ball}\\
			& \int_\S n(x_{n+1}+\cos\theta\<\nu,E_{n+1}\>) -H\<X,\nu\>dA=0.\label{Minkowski-Ball}
		\end{align} 	
	\end{proposition}
	\begin{proof} We prove it for completeness.
		Since $$\bar{\rm div} E_{n+1}=0,$$ by using integration by parts in $\O$, we get \eqref{Conservation-Ball}.
		Since $${\rm div}(E_{n+1}^T)=-H\left<\nu, E_{n+1}\right>,$$ by using integration by parts in $\S$ and \eqref{munu}, we get \eqref{Balancing-Ball}.
		Since $${\rm div}(X^T)=nV-H\left<X, \nu\right>,$$ by using integration by parts in $\S$, we get
		\begin{align}\label{xxeq1-b}
			\int_\S nV -H\<X,\nu\>dA=\int_{\Gamma}\<X, \mu\> ds.
		\end{align} 
	On the other hand, by \eqref{AS16-(2.3)} and \eqref{munu}, and the fact $\bar N(x)=x$, we obtain
	\begin{align}\label{xxeq2-b}
	    n\int_\Sigma\left<\nu,E_{n+1}\right>dA
	    =&\int_\Gamma\left(\left<x,\mu\right>\left<E_{n+1},\nu\right>-\left<x,\nu\right>\left<E_{n+1},\mu\right>\right)ds\notag\\
	    =&\int_\Gamma\left<E_{n+1},\bar\nu\right>ds=-\int_\Gamma\left<X,\bar\nu\right>ds=-\frac{1}{\cos\theta}\int_\Gamma\left<X,\mu\right>ds.
	\end{align}
		Here we invoke the definition of $X$ in the third equality; the last equality is derived again by \eqref{munu}. 
		
	
		Identity \eqref{Minkowski-Ball} follows from \eqref{xxeq1-b} and \eqref{xxeq2-b}.
	\end{proof}

\begin{proof}[{\bf Proof of \cref{Thm-HK-halfball}}]
	 Let $f$ be the solution to \eqref{MixedBdryPb-halfball2} with 
	\begin{eqnarray}\label{c2}
		c=-\frac{n}{n+1}\cos\theta\frac{  \int_TVdA}{\int_\Gamma\left<E_{n+1},\mu\right>ds}.
	\end{eqnarray}
Arguing as the proof of \cref{Thm-HK-halfspace}, we find that 
\begin{align*}
	\vert\bar{\nabla}^2f\vert\in L^2(\Omega),\quad\vert\nabla^2f\vert\in L^1(T).
\end{align*}
Using generalized Reilly's formula \eqref{qx} for $f$ and we obtain
\begin{align}\label{eq-Thm1.2-1}
	&\frac{n}{n+1}\int_\Omega Vdx
	=\frac{n}{n+1}\int_\Omega V\left(\bar{\De}f\right)^2dx\notag\\
	\geq&\int_\Sigma HVf_\nu^2dA
	+c\int_TV\left(\De f-\frac{\De V}{V}f\right)ds
	+nc^2\int_TVds\notag\\
	=&\int_\Sigma HVf_\nu^2dA
	+c\int_T\left(V\De f-\De Vf\right)ds
	-nc^2\int_\Sigma \left<E_{n+1},\nu\right>ds.
\end{align}
Here in the inequality, we used the the boundary condition $f=0$ on $\S$ and $V{\bar\n_{\bar N}} f-f\bar\n_{\bar N} V=cV$ on $T$, as well as the fact that the mean curvature of the half sphere is a constant, i.e., $H^T\equiv n$; for the last equality, we used \eqref{Conservation-Ball}.

Let us keep track of the second term, since $$\vert\nabla^2f\vert\in L^1(T),$$ by using the Green's second identity, we find
\begin{align*}
	c\int_T\left(V\De f-\De Vf\right)ds&=c\int_\Gamma V\left<\nabla f,\overline{\nu}\right>ds\\&=\frac{c}{\cos\theta}\int_\Gamma V\left<\bar\nabla f-\left<\bar\nabla f,\overline{N}\right>\overline{N},\mu\right>ds,
\end{align*}
here we used $f=0$ on $\Sigma$ in the first equality; for the last equality, we use the constant contact angle condition \eqref{munu}.

Notice that $f\in C^{1,\alpha}(\Omega)$, the Dirichlet condition and the Robin condition in \eqref{MixedBdryPb-halfball2} yield
\begin{align}\label{eq-Thm1.2-2}
	\int_\Gamma V\left<\bar\nabla f-\left<\bar\nabla f,x\right>x,\mu\right>ds=&-\int_\Gamma V\bar\nabla_{\bar N}f\left<x,\mu\right>ds\notag\\
	=&-c\int_\Gamma V\left<x,\mu\right>ds=-c\int_\Gamma\left<X+E_{n+1},\mu\right>ds,
\end{align}
here in the second equality we have used the fact that $\Gamma\subset\Sigma$ and hence $f=0$ on $\Gamma$; in the last equality we exploited that $\Gamma\subset T\subset\p\mathbb{B}^{n+1}$ and hence $X(x)=Vx-E_{n+1}$. 

By virtue of  \eqref{eq-Thm1.2-2}, \eqref{Balancing-Ball} and \eqref{xxeq2-b}, \eqref{eq-Thm1.2-1} thus reads
\begin{align}
	\frac{n}{n+1}\int_\Omega Vdx
	\geq\int_\Sigma HVf_\nu^2dA-\frac{c^2}{\cos\theta}\int_\Gamma \left<\mu,E_{n+1}\right>dA.
\end{align} 
On the other hand, by using the divergence theorem \cite[Lemma 2.1]{PT20},
\begin{align}
	\int_\Omega Vdx
	=\int_\Omega \left(V\bar\De f-\bar\De Vf\right)dx
	=&\int_\Sigma Vf_\nu dA+c\int_TVdA.
\end{align}

Invoking the choice of $c$ in \eqref{c2}, we find
\begin{align}\label{eq-Thm1.2-4}
	\int_\Sigma Vf_\nu dA
	=\int_\Omega Vdx+\frac{n}{n+1}\frac{\cos\theta\left(\int_T VdA\right)^2}{\int_\Gamma\left<E_{n+1},\mu\right>ds}
	\geq\frac{n+1}{n}\int_\Sigma HVf_\nu^2dA.
\end{align}
Using the H\"older inequality
\begin{align}
	\left(\int_\Sigma Vf_\nu dA\right)^2\leq\int_\Sigma HVf_\nu^2dA\int_\Sigma\frac{V}{H}dA,
\end{align}
in \eqref{eq-Thm1.2-4}, we obtain
\begin{align}
	\int_\Sigma\frac{V}{H}dA
	\geq\frac{n+1}{n}\int_\Omega Vdx+\cos\theta\frac{\left(\int_T VdA\right)^2}{\int_\Gamma\left<E_{n+1},\mu\right>ds}.
\end{align}
This shows \eqref{EQ-HK-halfball2}. Inequality \eqref{EQ-HK-halfball} is equivalent to \eqref{EQ-HK-halfball2} in view of \eqref{Conservation-Ball} and \eqref{Balancing-Ball}.

The characterization of equality case in \eqref{EQ-HK-halfball} or \eqref{EQ-HK-halfball2} follows from the proof of \cref{Thm-HK-halfspace}, this completes the proof.
\end{proof}

Consequently, the Alexandrov type result \cref{Theorem2-ball} is set up. Thanks to the CMC condition, the Minkowski type formula \eqref{Minkowski-Ball} shows that the Heintze-Karcher type inequality \eqref{EQ-HK-halfball} must be an equality, and it follows from \cref{Thm-HK-halfball} that $\Sigma$ must be a $\theta$-capillary spherical cap.

\appendix
\section{A Fredholm alternative for the mixed boundary value problem}\label{Appendix}
{\color{black}The purpose of the appendix is to present a detailed statement and proof of the Fredholm alternative for mixed boundary elliptic equation, which was brought up in \cite{Liebermann86} without proof.}

 {\color{black}For completeness}, we present some existence and regularity results for mixed boundary value problems
 \begin{align}\label{MBVP}
		\begin{cases}
			\bar\Delta f=h \quad&\text{in }\O,\\
			f=0 \quad&\text{on }\S,\\
			\bar\n_{\bar N} f+ f=g &\text{on }{\rm int}(T),	\end{cases}	
	\end{align}
  which was proved by Lieberman \cite{Liebermann86,Liebermann89}.
  
  To facilitate the presentation, We recall the definition of weighted H\"older spaces. Set $d_\Gamma(x)=dist(x,\Gamma)$, $\Omega_\delta=\{x\in\Omega:d_\Gamma(x)>\delta\}$. For $a\geq 0$, $b\geq-a$, we define 
\begin{align*}
    {|f|}_a^{(b)}=\sup_{\delta>0}\delta^{a+b}{|f|}_{a;\Omega_\delta},
\end{align*}
where $|f|_{a;\Omega_\delta}$ is the standard norm on $\Omega_{\delta}$. We denote by $H_a^{(b)}$ the set of all functions $f$ on $\Omega$ with finite norm $|f|_a^{(b)}$.

\begin{theorem}\label{ThmA1}
There exists a solution $f\in C^{2}( \O\cup{\rm int} (T))\cap C^0(\bar \O)$ to \eqref{MBVP} for all $h\in C^{\alpha}(\O\cup {\rm int} (T))$, $g\in C^{1,\alpha}(\O\cup {\rm int} (T))$. 
    
\end{theorem}
\begin{proof}
See \cite[Theorem 1]{Liebermann86}.
\end{proof}
\begin{theorem}\label{regularity thm}
Assume $\theta\in (0,\frac{\pi}{2})$. Let $f\in C^{2}( \O\cup{\rm int} (T))\cap C(\bar \O)$ be a solution to \eqref{MBVP}. Then for any $\lambda\in \left(1,\frac \pi {2\theta}\right)$ and any noninteger $a>2$, we have
\begin{align}\label{regularity ineq}
 |f|_a^{(-\lambda)}\leq C(|h|_{a-2}^{(2-\lambda)}+|g|_{a-1}^{(1-\lambda)}+|f|_0).  
\end{align}
\end{theorem}
\begin{proof}
 See \cite[Theorem 4]{Liebermann89}. Here we only explain the admissible range $\left(1,\frac \pi {2\theta}\right)$ for $\lambda$, which is not explicitly expressed in \cite[Theorem 4]{Liebermann89}. 
 
 Locally near every $x_0\in\Gamma$, up to a transformation, there exists a Cartesian coordinate $(x^1,\ldots,x^{n+1})$, centered at $x_0$, such that the {\color{black}corresponding} cylindrical coordinates $(r,\eta,x')$ is given by 
 $$ x_1=r\cos\eta,\quad  x_2=r\sin\eta,\quad x'=(x^3,\ldots,x^{n+1}),$$
 with $0\leq \eta\leq \theta$.\footnote{The coordinate representation follows from the definition of the wedge condition by Lieberman in \cite{Liebermann86}, which is clearly satisfied by domains that we consider.}
 
 The key of the proof in \cite[Theorem 4]{Liebermann89} is to find a Miller-type barrier function $\psi=r^\lambda \varphi (\eta)$ satisfying  
 \begin{align}\label{miller}
     \begin{cases}
         \bar\Delta \psi=r^{\lambda-2}(\lambda^2\varphi+\varphi_{\eta \eta})\leq-c_1r^{\lambda-2}\quad&\text{for }\eta\in (0,\theta),\\
			\bar\n_{\bar N} \psi=-r^{\lambda-1}\varphi_{\eta }\geq c_1 r^{\lambda-1} &\text{for } \eta=0,\\
			c_1\leq \varphi\leq 1 &\text{for } \eta\in [0,\theta],
     \end{cases}
 \end{align}
 with a positive constant $c_1$. The Miller-type barrier $\varphi(\eta)$ is constructed from a perturbation of $\tilde{\varphi}(\eta)=\cos(\lambda\eta)$, which has a positive lower bound in $[0, \theta]$ and satisfies \eqref{miller} only if $\theta\in (0,\frac{\pi}{2\lambda})$. Hence to ensure $\lambda>1$, one needs to restrict $\theta\in (0,\frac{\pi}{2})$.
 
\end{proof}
\begin{lemma}
If $f\in C^{2}( \O\cup{\rm int} (T))\cap C(\bar \O)$ is a solution to \eqref{MBVP}. Then there exists a {\color{black}positive }constant $C$ such that
\begin{align}\label{c0 estimate}
 |f|_0\leq C(|h|_0+|g|_0).   
\end{align}
\begin{proof}
 Set $C_1=|f|_0+|g|_0$. For the case $B=\bar \rr^{n+1}_+$, assume without loss of generality that $\Omega\subset\left\{0<x_{n+1}<R\right\}$ for some $R$ large, consider the functions $v_1(x)=C_1(e^{x_{n+1}}-e^R)$. A direct computation then yields, 
	\begin{align}\label{regularity ineq}
    \begin{cases}
	        \bar\Delta v_1\geq C_1 \quad&\text{in }\O,\\
			v_1\leq0 \quad&\text{on }\S,\\
			\bar\n_{\bar N} v_1+v_1\leq -C_1 &\text{on } {\rm int}(T).
	   \end{cases}
	\end{align}
	Applying the maximum principle (see for example \cite[Lemma 4.1]{Tang13}) to $f+v_1$ and $-f+v_1$, we obtain \eqref{c0 estimate}.
	
	For the case $B=\bar{\mathbb{B}}^{n+1}$, 
	we choose $v_2=C_1(|x|^2-4)$ to replace $v_1$, and the process follows similarly.
	
\end{proof}
\end{lemma}
We can now derive a Fredholm alternative for the mixed boundary value problem.
\begin{theorem}\label{fredholm lemma}
Assume $\theta\in(0,\frac\pi 2  )$.  Let $\gamma$ be a nonnegative integer, $a$ be a noninteger greater than 2, $\lambda$ be defined as in  \cref{regularity thm}.  Then either $(a)$ the homogeneous problem
\begin{align}
		\begin{cases}
			\bar\Delta f=0 \quad&\text{in }\O,\\
			f=0 \quad&\text{on }\S,\\
			\bar\n_{\bar N} f-\gamma f=0 &\text{on }{\rm int}(T),	\end{cases}	
	\end{align}
has nontrivial solutions; or $(b)$ the homogeneous problem 
has only the trivial solution, in which case, for all $f\in H_{a-2}^{(2-\lambda)}$, $g\in H_{a-1}^{1-\lambda}$ the inhomogeneous problem 
\begin{align}\label{inhomog eq}
		\begin{cases}
			\bar\Delta f=h \quad&\text{in }\O,\\
			f=0 \quad&\text{on }\S,\\
			\bar\n_{\bar N} f-\gamma f=g &\text{on }{\rm int}(T),	\end{cases}	
	\end{align}
has a unique solution $f\in H_a^{(-\lambda)}$.
\end{theorem}
\begin{remark}\label{remark-appendix}
\normalfont
By the monotonicity of norm we have $|f|_\lambda^{(-\lambda)}\leq |f|_a^{(-\lambda)}$.  Therefore
if there is a solution $f\in H_{a}^{(-\lambda)}$ to problem \eqref{inhomog eq},
then the solution $f$ is also in $C^{1,\alpha}(\bar\O)$ with $\alpha=\lambda-1$.
\end{remark}
\begin{proof}[Proof of \cref{fredholm lemma}]
{\color{black}We follow the classical proof of the Fredholm alternative for the Dirichlet problem but with a careful choice of function space.} Set $\mathcal{A} =\{u\in H_a^{(-\lambda)}: f=0\text{ on }\Sigma\}$, $\mathcal{B}= H_{a-2}^{(2-\lambda)}\times H_{a-1}^{(1-\lambda)}$.
By \cref{ThmA1}, \cref{regularity thm} and \eqref{c0 estimate},
there exists a 
solution $u\in\mathcal{A}$ to problem \eqref{MBVP} for all $(h,g)\in \mathcal{B}$ .
 Uniqueness of the solution to problem \eqref{MBVP} follows from the maximum principle immediately. Let  $Q:\mathcal{B}\rightarrow \mathcal A$ be an operator such that $Q(h,g)$ is the unique solution of problem \eqref{MBVP}.  One can readily see that the operator $Q$ is well-defined and bijective.
Notice that the inhomogeneous problem \eqref{inhomog eq} is equivalent to the following equation
\begin{align}\label{equi inhomog}
f-Q(0,(1+\gamma )f)=Q(h,g).
\end{align}

Thus $f\in\mathcal A$ is a solution to \eqref{equi inhomog} if and only if it solves \eqref{inhomog eq}.
Let $P$ be an operator from $H_{a-1}^{(1-\lambda)}$ to itself such that $Pu=Q(0,(1+\gamma )u)$ for all $u\in H_{a-1}^{(1-\lambda)}$.
Then the equation \eqref{equi inhomog} reads as 
\begin{align}\label{fredholm eq}
    f-Pf=v,
\end{align}
where $v=Q(h,g)$.

 To apply the classical Fredholm alternative(see e.g., \cite[Theorem 5.11]{GT01}) for \eqref{fredholm eq}, we need to verify that $P$ is a compact operator. {\color{black}To this end}, let $\{g_k\}$ be a bounded sequence in $H_{a-1}^{(1-\lambda)}$.
 By \cref{ThmA1} and \eqref{c0 estimate}, there exists $f_k\in \mathcal A$ such that $Pg_k=f_k$, and we have
 \begin{align}
   |f_k|_a^{(-\lambda)}\leq C\left(|g_k|_{a-1}^{(1-\lambda)}+|f_k|_0\right).   \end{align}
 Since $\left\{|f_k|_0\right\}_k$ is bounded, we have $\left\{f_k\right\}_k$ is also bounded in $H_{a}^{(-\lambda)}$.
 By virtue of the Ascoli-Arzela theorem, up to a subsequence, there exists ${\color{black}\tilde{f}}\in H_{a-1}^{(1-\lambda)}$, such that $f_{k_j}$ converges to ${\color{black}\tilde{f}}$ in $H_{a-1}^{(1-\lambda)}$. Thus $P$ is compact and the classical Fredholm alternative applies: \eqref{fredholm eq} has a unique solution $f\in H_{a-1}^{(1-\lambda)}$, provided that the homogenous equation $f-Pf=0$ has only the trivial solution $f=0$. Since $Q$ maps $\mathcal B$ onto $\mathcal A$, any solution  $f\in H_{a-1}^{(1-\lambda)}$ of \eqref{fredholm eq} also belongs to $\mathcal A$. This completes the proof.

\end{proof}

\begin{remark}
We remark that, in the case $\theta=\frac{\pi}{2}$, we can use boundary reflection to get better regularity, say global $W^{2, p}$ estimate for any $p$, see for example \cite[Proposition 3.5]{GX19}.
\end{remark}


\printbibliography
		
\end{document}